\documentclass[12pt]{amsart}

\usepackage[top=3.4cm,bottom=3.4cm,inner=4cm,outer=4cm]{geometry}
\usepackage{tikz-cd}

\usepackage{comment}

\newcommand{\disp}{\displaystyle}
\newcommand{\nc}{\newcommand}
\nc{\G}{{\Gamma}} \nc{\BC}{{\mathbb C}} \nc{\BQ}{{\mathbb Q}}
\nc{\BR}{{\mathbb R}} \nc{\BZ}{{\mathbb Z}} \nc{\BP}{{\mathbb P}} \nc{\PC}{{\BP_1(\BC)}}
\nc{\BN}{{\mathbb N}} \nc{\BM}{{\mathbb M}}
\nc{\fH}{{\mathbb H}}

\nc{\mat}{{\binom{a\,\ b}{c\,\ d}}}
\nc{\U}{{\mathcal U}}
\nc{\PS}{{\mbox{PSL}_2(\BZ)}} \nc{\SL}{{\mbox{SL}_2(\BZ)}}
\nc{\SR}{{\mbox{SL}_2(\BR)}} \nc{\PR}{{\mbox{PSL}_2(\BR)}}
\nc{\GL}{{\mbox{GL}}} \nc{\PQ}{{\mbox{PGL}_2^+(\BQ)}}
\nc{\GR}{{\mbox{GL}_2^+(\BR)}} \nc{\PG}{{\mbox{PGL}_2(\BC)}}
\nc{\GC}{{\mbox{GL}_2(\BC)}}
\nc{\f}{{\mathcal{F}(\fH)}}
\nc{\Cc}{\widehat{\BC}}
\nc{\e}{{E_{\rho}(\G)}}
\nc{\g}{{\gamma}}
\nc{\vm}{{V_{\rho}(\G)}}
\nc{\oo}{{\mathcal O}}
\nc{\M}{{\mbox{M}}}
\nc{\om}{{\omega}}
\nc{\Om}{{\Omega}}
\nc{\TX}{{\widetilde{X}}}
\nc{\ol}{\overline}
\nc{\cl}{{\mathcal L}}
\nc{\ce}{{\mathcal E}}
\nc{\la}{{\lambda}}
\nc{\La}{{\Lambda}}
\nc{\cz}{{\mathcal Z}}

\newtheorem{numbered}{}[section]
\newtheorem{thm}[numbered]{Theorem}

\newtheorem{lem}[numbered]{Lemma}
\newtheorem{remark}[numbered]{Remark}
\newtheorem{prop}[numbered]{Proposition}
\newtheorem{cor}[numbered]{Corollary}

\numberwithin{equation}{section}

\newcommand{\thmref}[1]{Theorem~\ref{#1}}
\newcommand{\propref}[1]{Proposition~\ref{#1}}
\newcommand{\secref}[1]{\S\ref{#1}}

\begin{document}

\title[]{Automorphic Schwarzian equations}
\author[]{Abdellah Sebbar} \author[]{Hicham Saber}
\address{Department of Mathematics and Statistics, University of Ottawa, Ottawa Ontario K1N 6N5 Canada}
\address{Department of Mathematics, University of Ha'il, Kingdom of Saudi Arabia}
\email{asebbar@uottawa.ca}
\email{hicham.saber7@gmail.com}


\begin{abstract}
	This paper concerns the study of the Schwarz differential equation
	 $\{h,\tau \}=s\,E_4(\tau)$ where $E_4$ is the weight 4 Eisenstein series and $s$ is a complex parameter. In particular, we determine all the values of $s$ for which the solutions $h$  are modular functions for a finite index subgroup of $\SL$. We do so using the theory of equivariant functions on the complex upper-half plane as well as an analysis of the representation theory of $\SL$. This also leads to the solutions to the Fuchsian differential equation $y''+s\,E_4\,y=0$. 
\end{abstract}
\maketitle
\tableofcontents
\section{Introduction}

Let $D$ be a domain in $\BC$ and $f$ a meromorphic function on $D$. The Schwarz derivative, or the Schwarzian of $f$, is defined by
\begin{equation}\label{schwarz}
\{f,z\}\,=\, \left(\frac{f''}{f'}\right)'-\frac12 \left(\frac{f''}{f'}\right)^2\,=\,
\frac{f'''}{f'}-\frac32\left(\frac{f''}{f'}\right)^2.
\end{equation}
It was named after Schwarz by Cayley. However,  Schwarz himslef pointed out that it was discovered by Lagrange in 1781 and it also
appeared in a paper by Kummer in 1836 ~\cite{ov-ta}.

The Schwarz derivative plays an important role in the study of the complex projective line, univalent functions, conformal mapping, Teichmuller spaces
and in the theory of automorphic functions and hypergeometric functions \cite{ahl,duren, ford,mc-se,nehari, s-g}. Most importantly, there is a strong link with  the theory of ordinary differential equations as follows.

Let $R(z)$ be a meromorphic function on $D$ and consider the second order differential equation
\[
y''+\frac{R(z)}{2}\,y=0
\]
with two linearly independent solutions $y_1$ and $y_2$. Then $f=y_1/y_2$ is a solution to the Schwarz differential equation
\[
\{f,z\}\,=\,R(z).
\]
This connection with second order ordinary differential equations leads to the most  important properties 
of the Schwarz derivative, see \secref{mde}.

There is also an important link with automorphic functions for a discrete subgroup $G$ of $\PR$, that is a Fuchsian group of the first kind acting on the upper half-plane $\fH=\{\tau\in\BC \mid \Im(\tau)>0\}$ by linear fractional transformations
\[
\gamma \tau=\frac{a\tau+b}{c\tau+d}\,,\ \ \gamma=\binom{a\,\ b}{c\,\ d}\in \G.
\]
Let $f$ be an automorphic function for $G$, that is $f$ is a meromorphic function on $\fH$ invariant under the action of $G$  (in addition to a meromorphic behavior at the cusps). Using the projective invariance of the Schwarz derivative and the fact that it defines a quadratic differential, we see that  $\{f,\tau\}$ is a weight 4 automorphic form for $G$. Moreover, if $G$ is genus 0 and $f$ is a Hauptmodul, then $\{f,\tau\}$ is an automorphic form for the normalizer of $G $ in $\PR$ ~\cite{mc-se}.
As an example, let $\lambda$ be the Klein elliptic modular function for $\Gamma(2)$, then 
\[
\{\lambda,\tau\}\,=\,\frac{\pi^2}{2}E_4(\tau),
\]
where $E_4$ is the weight 4 Eisenstein series.
However, for a given weight 4 automorphic form $F$ for the Fuchsian group $G $, a meromorphic function $f$ satisfying $\{f,\tau\}=F$ is not necessary an automorphic function. Nevertheless, it defines an interesting class of functions. Indeed, there exists a representation $\rho$ of $G$ into $\GC$ (or $\PG$) such that
for all $\gamma\in G $ and $\tau\in\fH$
\[
f(\gamma\cdot \tau)=\rho(\gamma)\cdot f(\tau),
\]
where the action in both sides is by linear fraction and $f$ is called $\rho-$equivariant. When $\rho=1$ is trivial, then $f$ is an automorphic function for $G $. In addition, the form $F$ is holomorphic on $\fH$ if and only if  the solution $f$ is locally univalent on $\fH$, meaning that its derivative is nowhere vanishing.

We focus on the case $G =\SL$ and $F$ being a holomorphic weight 4 modular form.  In other words, we look at the equation
\begin{equation}
\{h,\tau\}\,=\,s\,E_4(\tau),
\end{equation}
where $s$ is a complex parameter. If a solution $h$ is given, then any other solution is simply a linear fraction of $h$, and their corresponding representations are conjugate by a constant matrix in $\GC$.
Thanks to the Frobenius method for the Fuchsian differential equation, one can write down the shape of the solution $h$. In particular $h$ will have a logarithmic singularity at infinity if and only if $s=2\pi^2n^2$ for an integer $n$. Our approach to study the general solutions consists of analyzing their corresponding representation. Indeed, $h$ will be a modular function if and only if its invariance group $\G =\ker \rho$ has a finite index in $\SL$, or equivalently, Im$\, \rho$ is finite. We distinguish between the irreducible representations and the reducible ones. In the former case, we find that there are only finitely many  possibilities for the level of $\G$ defined to be the least common multiple of the cusp widths.   As a solution $h$ is necessarily locally univalent, $\G$ is a torsion-free, normal and finite index subgroup of $\SL$ and  $h$ realizes a covering map $X(\G)\longrightarrow\PC$ where $X(\G)$ is the modular curve attached to $\G$. Finally, applying the Riemann-Hurwitz to this covering, we deduce that the genus of $\G$ is zero which leads to a finite number of possibilities. 

On the other hand, if $\rho$ is reducible, it turns out that $\G$ cannot have finite index in $\SL$ and so does not provide any modular solution. This  leads to the main result in this paper stating that there are infinitely many values for the parameter $s$ that lead  to a modular function $h$ as a solution to $\{h,\tau\} =sE_4(\tau)$. This happens precisely when $s$ has the form $\disp s=2\pi^2(n/m)^2$ where $m$ and $n$ are positive integers
 with $2\leq m\leq 5$ and $\gcd(m,n)=1$. Here $m$ becomes the level of $\ker\rho$, $n$ is the ramification index at $\infty$ of the covering $h$ and together, with the degree  $d$ of this covering, they  satisfy
 \[
 (6n-m)\nu_{\infty}=12d,
 \]
 where $\nu_{\infty}$ is the parabolic class number of $\G$.   It is worth noting that four of our solutions corresponding to $d=1$ already appeared in \cite{ss2} under different circumstances.

Furthermore, as a solution $h$ to $\{h,\tau\}=sE_4$ has a nowhere vanishing derivative $h'$  on $\fH$,  it is possible to provide a basis of global solutions of the Fuchsian differential equation $\disp y''+\frac{s}{2}E_4\,y=0$ by means of $y_1=h/\sqrt{h'}$ and $y_2=1/\sqrt{h'}$.

Second order modular differential equations have been the subject of study by several authors  \cite{mason-franc,  ka-ko,   ka-za,  milas, mukh,  yang} among others in the context of hypergeometric equations, Virasoro and vertex algebras, and rational conformal theories. In these papers, the equation takes the form
\[
y''+\alpha\, E_2\,y'+\beta\,E_4\,y=0,
\]
with specific values of $\alpha$ and $\beta$.
Our approach is different in the sense that our Fuchsian differential equations do not  contain a term in $y'$ with a quasi-modular coefficient simply because they arise from a Schwarzian differential equation which is our main interest. The differential equations can rather be looked at as Schr\"{o}dinger equation with an automorphic potential in the same spirit that a Lam\'e equation has an elliptic potential.

The paper is organized as follows. In chapter 2 we introduce the notion of equivariant function attached to a two-dimensional complex representation of a Fuchsian group of the first kind. We make explicit how they are closely connected to vector-valued automorphic forms with multiplier system given by the above representation. In chapter 3, we list the main properties of the Schwarz derivative as well as its automorphic properties. We explain how solutions to Schwarzian equations are equivariant functions. In chapter 4, we focus on the modular group, and we show that it does not matter in our context whether we work with representations of $\SL$ or of $\PS$. We then use a result of G. Mason to classify the representations of the modular group having a finite image. In chapter 5, we apply the Frobenius method to our Fuchsian differential equations after a change of variable to find explicit expansions of the solutions of both the Fuchsian differential equations and the Schwarz differential equation. In chapter 6, we show that if there is a solution unramified at infinity, then the kernel of the attached representation is a normal, torsion-free and genus zero congruence subgroup of the modular group. This allows us to explicitly write down solutions in terms of classical modular forms and functions. In chapter 7, we determine when a Schwarzian equation admits solutions that are modular functions in the general case while assuming that the representation is irreducible. Finally, in chapter 8, we show that the case of reducible representation cannot lead to a modular function as a solution to the Schwarz differential equation.

\section{Automorphic forms and Equivariant functions}\label{sec-equiv}

Let $\G$ be a Fuchsian group of the first kind   and let $\rho$ be a finite dimensional representation of $\G$, that is, a homomorphism of $\G$ into $\GL_n(\BC)$, $n\in\BZ_{>0}$. A $n-$dimensional vector-valued automorphic form of weight $k\in \BZ$ and multiplier system $\rho$ is a vector $F(\tau)=(f_1,\ldots f_n)^t$ where $f_1,\ldots ,f_n$ are  meromorphic functions on $\fH$ satisfying 
\begin{equation}\label{vaf}
F(\gamma \tau)\,=\,(c\tau+d)^k\,\rho(\gamma)\,F(\tau)\,,\ \tau\in\fH\,,\ \gamma=\binom{a\ \ b}{c\ \ d}\in\G\,.
\end{equation}
When $n=2$, we define for the pair $(\G,\rho)$  the notion of $\rho-$equivariant function. It is a meromorphic function $h$ on $\fH$ satisfying
\begin{equation}\label{rho-equiv}
h(\gamma \tau)\,=\,\rho(\gamma)\;h(\tau)\,,\ \tau\in\fH\,,\ \gamma\in\G\,,
\end{equation}
where on both sides the action of the matrices is by linear fractional transformation. When $-I_2\in \G$, for this  definition to make sense, we need to have $\rho(-I_2)\in {\mathbb C}^* I_2$. Thus $\rho$ induces a representation $\bar{\rho}$ of the image of $\G$ in $\PS$ into $\PG$, and the same definition for equivariance will be used interchangeably between homogeneous and inhomogeneous Fuchsian groups.

When $\rho$ induces a character of $\Gamma$, i.e. when $\rho(\G)\subseteq{\mathbb C}^* I_2$, then $h$ is a scalar automorphic function for $\Gamma$ (with a multiplier system). When $\rho$ is the trivial (or the  standard  representation) of $\G$, that is, for all $\gamma\in\Gamma$, $\rho(\gamma)=\gamma$, then $h$ is simply called an equivariant function for $\Gamma$. They appeared in \cite{brady} and have been extensively studied from the automorphic, equivariant and elliptic points of view with interesting applications in \cite{ell-zeta,  sb2, ss3, s-s2, ss1}. Examples of such functions are provided as follows: Let $f$ be a weight $k$ automorphic form for $\Gamma$, then 
\begin{equation}\label{rational-equiv}
h_f(\tau)\,=\,\tau\,+\,k\,\frac{f(\tau)}{f'(\tau)}
\end{equation}
is an equivariant function for $\Gamma$ and equivariant functions arising in this way are called rational  \cite{sb1}.

For a given pair $(\G,\rho)$, there is a close connection between 2-dimensional vector-valued automorphic forms of multiplier system $\rho$ and $\rho-$equivariant functions for $\G$. Indeed, if $F=(f_1,f_2)^t$ is a vector-valued automorphic form of multiplier $\rho$ and arbitrary weight, then one can easily verify that $h=f_1/f_2$ is $\rho-$equivariant. A less trivial fact is that every $\rho-$equivariant function arises in this way \cite[Theorem 4.4]{s-s2}. Another nontrivial fact is that for an arbitrary pair $(\G,\rho)$ (even when $\G$ is a Fuchsian group of the second kind), vector-valued automorphic  forms exist in any dimension \cite[Theorem 6.6]{ss4}. As a consequence, $\rho-$equivariant functions always exist for arbitrary $\G$ and an arbitrary 2-dimensional representation $\rho$ of $\G$.

\section{Automorphic differential equations}\label{mde}

The following properties of the Schwarz derivative will be very useful.
The functions involved are all meromorphic functions on the domain $D$.
\begin{itemize}
	\item Projective invariance:
	\begin{equation}	\left\{\frac{af+b}{cf+d}\,,z\right\}\,=\,\{f,z\}\;,\ \mbox{for }
	\binom{a\ \,b}{c\ \,d}\in\GC. \label{sd1}
	\end{equation}
	\item Cocycle property: If $w$ is a function of $z$, then
	\begin{equation}\label{sd2}
	\{f,z\}\,=\,\{f,w\}(dw/dz)^2+\{w,z\}, 
	\end{equation}
	which means that $\{f,z \}$ is a quadratic differential
	\[
	\{f,z\}dz^2\,=\,\{f,w\}dw^2+\{w,z\}dz^2. 
	\]
	\item 
	\begin{equation} \label{sd3} \{f,z\}\,=\,0 \mbox{ if and only if }  f(z)=\frac{az+b}{cz+d}\, \mbox{ for some } a,\,b,\,c,\,d\in\BC
	\end{equation}.
	\item If $\displaystyle w=\frac{az+b}{cz+d}$ with 
	$\displaystyle \binom{a\,\ b}{c\,\ d}\in \GC $, then 
	\begin{equation}
	\{f,z\}\,=\,\{f,w\}\,
	\frac{(ad-bc)^2}{(cz+d)^4}. \label{sd4}\end{equation}
	\item For two meromorphic functions $f$ and $g$ on $D$, 
	\begin{equation}
	\{f,z\}\,=\,\{g,z\}\,\mbox{ if and only if }\, f(z)\,=\,\frac{ag(z)+b}{cg(z)+d}\, \mbox{ for some }  \binom{a\,\, b}{c\,\, d}\in \GC. \label{sd5}
	\end{equation}	
	\item If $w(z)$ is a function of $z$ with $w'(z_0)\neq 0$ for some $z_0\in D$, then in a neighborhood of $z_0$, we have
	\begin{equation} 
	\{z,w\}\,=\,\{w,z\}\,(dz/dw)^2\,. \label{sd6}
	\end{equation}
\end{itemize}
These properties are a consequence of the link with the second order linear differential equations.

Throughout this section, $\G$ is a Fuchsian group of the first kind and $f$  is a weight 4 automorphic form for $\G$.  All the forms and functions are meromorphic unless otherwise stated. We look at  the Schwarz differential equation
\begin{equation}\label{schP}
\{h,\tau\}\,=\,f(\tau).
\end{equation} 
and the related Fuchsian differential equation
\begin{equation}\label{diffP}
y''\,+\,\frac{f}{2}\,y\,=\,0
\end{equation}

\begin{prop}\label{prop1}
	Let $h(\tau)$ be a meromorphic function on $\fH$, then $\{h,\tau\}=f(\tau)$ for a weight 4 automorphic form $f$ for $\G$ if and only if $h$ is $\rho-$equivariant for some 2-dimensional representation $\rho$ of $\G$.
\end{prop}
\begin{proof}
Suppose that $f$ is a weight 4 automorphic form $f$ for $\G$ and $\{h,\tau\}=f$.
Using the properties of the Schwarz derivative we have, for $\displaystyle \gamma=\mat\in\Gamma$,
\begin{align*}
(c\tau+d)^4\,f(\tau)\,&=\,
f\left(\frac{a\tau+b}{c\tau+d}\right)\\
&=\,	\left\{h\left(\frac{a\tau+b}{c\tau+d}\right),\frac{a\tau+b}{c\tau+d}\right\}\\
&=\,(c\tau+d)^4\,\left\{h\left(\frac{a\tau+b}{c\tau+d}\right),\tau\right\}\quad\mbox{ using } \,\eqref{sd4}.
\end{align*}
Therefore,
\[
\{h,\tau\}\,=\,\left\{h\left(\frac{a\tau+b}{c\tau+d}\right),\tau\right\}.
\]
Hence, according to \eqref{sd5}, there exists $\displaystyle \binom{A\,\ B}{C\,\ D}\in\GC$ such that
\[
h\left(\frac{a\tau+b}{c\tau+d}\right)\,=\,\frac{Ah(\tau)+B}{Ch(\tau)+D}.
\]
This defines a 2-dimensional representation $\rho$ of $\Gamma$ in $\PG$ such that $h$ is $\rho-$equivariant. Conversely, using the same properties \eqref{sd4} and \eqref{sd5}, it is easy to see that if $h$ is $\rho-$equivariant, then $f=\{h,\tau\}$ is a weight 4 automorphic form for $\G$.
	\end{proof}
In the meantime, if $h_1$ and $h_2$ are two solutions to $\disp \{h,\tau\}=f$, then there exists $\sigma\in\GC$ such that $h_2(\tau)=\sigma h_1(\tau)$. Let $\rho_1$ (resp. $\rho_2$) be the  2-dimensional
complex representation   of $\Gamma$ such that $h_1$ is $\rho_1 -$equivariant (resp. $h_2$ is $\rho_2  -$equivariant). It is easy to show the following
\begin{prop}\label{sigma}
	The representations $\rho_1$ and $\rho_2$ are conjugate. More precisely, \[\rho_2=\sigma \rho_1 \sigma^{-1}.\]
\end{prop}
\begin{thm}
	Suppose that $f$ is a weight $4$ holomorphic automorphic form for $\G$.
	\begin{enumerate}
		\item  If $y_1$ and $y_2$ are two linearly independent holomorphic solutions to $ y''+\frac{1}{2}\,f\,y=0$ then $\displaystyle F=\binom{y_1}{y_2}$ is a weight -1 vector-valued automorphic form for some multiplier system $\rho$. Furthermore, $h=y_1/y_2$ is a $\rho-$equivariant function satisfying $\{h,\tau\}=f$.
		\item If $h$ be a solution to $\{h,\tau\}=f$, then $y_1=h/\sqrt{h'}$ and $y_2=1/\sqrt{h'}$ are two linearly independent holomorphic solutions to $ y''+\frac{1}{2}\,f\,y=0$.
	\end{enumerate}

\end{thm}

\begin{proof} First, since $\fH$ is simply connected and $f$ is holomorphic, linearly independent global solutions to $y''+\frac{1}{2}\,f\,y=0$ always exist.
To prove (1), let $y$ be a holomorphic solution, then for $ \gamma=\binom{a\ b}{c\ d}\in\G$, a straightforward calculation shows that the function $y^*(\tau)=(c\tau+d)\;y(\gamma \tau)$ is also a holomorphic solution. Therefore,  if $y_1$ and $y_2$ are two linearly independent solutions, then $y_1^*$ and $y_2^*$ are also two linearly independent solutions. It follows that there exists a 2-dimensional representation $\rho$ of $\G$ such that $\displaystyle \binom{y_1^*}{y_2^*}=\rho(\gamma)\binom{y_1}{y_2}$. In other words, 
 $\displaystyle F=\binom{y_1}{y_2}$ is a weight -1 vector-valued automorphic form of multiplier $\rho$ (in particular $\rho(-I_2)=-I_2$). Moreover,  $h=y_1/y_2$ is a $\rho-$equivariant function satisfying $\{h,\tau\}=f$.
 
 As for (2), let $h$ be a solution to $\{h,\tau\}=f$, then according to the expression \eqref{schwarz}, if $h'(\tau_0)=0$, then $\tau_0$ is a double pole of $\{h,\tau\}$. But since $f$ is holomorphic everywhere, we see that $h'(\tau)$	is nowhere vanishing on $\fH$ (We choose the principal branch of the square root). In other words $h$ is locally univalent on $\fH$ and the square root $\sqrt{h'(\tau)}$ is defined everywhere as a meromorphic function. Furthermore, from \eqref{schwarz} again, we see that if $h$ has a multiple pole at $\tau_0$, then $\{h,\tau\}$ has a double pole at $\tau_0$ which is not the case since again $f$ is holomorphic everywhere. A simple analysis shows that if $\tau_0$ is a  pole of order 1 for $h$, then it is a regular point for its Schwarz derivative. We deduce that all the poles of $h$, if there are any,  are simple, and thus  they are double poles for $h'(\tau)$ and simple zeros for $1/\sqrt{h'(\tau)}$. We deduce that $y_1=h/\sqrt{h'}$ and $y_2=1/\sqrt{h'}$ are holomorphic on $\fH$.  Differentiating twice $y_1$ gives
\[
y_1''=-\frac12  h{h'}^{-1/2} \{h,\tau\}
\]
and so $y_1$ is a solution to $y''+\frac{1}{2}\,f\,y=0$ and the same is true for $y_2$.
\end{proof}

\section{Finite image representations}\label{reps}
If $\G$ is a subgroup of $\SL$, we denote by $\bar{\G}$
 the corresponding homogeneous subgroup of $\PS$ and let $\pi$ denote the natural surjection $\pi:\SL\longrightarrow \PS$ as well as $\pi:\GC\longrightarrow\PG$.
 If $\rho$ is a representation of $\SL$ such that $\rho(-I_2)\in{\BC^*I_2}$, then $\rho$ induces a representation $\bar{\rho}$ of $\PS$ in $\PG$ such that 
 \begin{equation}\label{cd}
 \pi\circ\rho=\bar{\rho}\circ\pi.
 \end{equation}
 The commutator group $\G'$ of $\SL$ is an index 12 normal congruence subgroup of $\SL$ of level 6 and $\bar{\G}'$, the commutator group of $\PS$ is a normal level 6 congruence subgroup of $\PS$ of index 6 \cite{rankin}. In addition, $\bar{\G}'$ has genus 1. Furthermore, the quotient $\SL/{\G'}$ (resp. $\PS/{\bar{\G }'}$) is a cyclic subgroup of order 12 generated by the class of  $T=\binom{1\ \ 1}{0\ \ 1}$ (resp. a cyclic subgroup of order 6 generated by the class of $T(\tau)=\tau+1$). It follows that the group of characters of $\SL$ (which are trivial on $\G'$) is cyclic generated by a character assigning to $T$ a primitive 12th  root of unity. In the case of $\PS$, the group of characters is cyclic of order 6 generated by any character assigning a primitive 6th root of unity to $T$.
 
 We now focus on representations of $\SL$ or $\PS$ having finite images or equivalently finite index kernels.
 \begin{prop}
 	Let $\rho$ be a representation of $\SL$ and let $\bar{\rho}$ be the induced representation of $\PS$. Then Im$\,\rho$ is finite if and only Im$\,\bar{\rho}$ is finite.
 \end{prop}
\begin{proof}
	Since $\bar{\rho}\circ\pi=\pi\circ\rho$, it is clear that if Im$\,\rho$ is finite then so is Im$\,\bar{\rho}$. We now suppose that Im$\,\bar{\rho}$ is finite. It follows that Im$\,\bar{\rho}\circ\pi$ is also finite and
	\[
	\mbox{Im}\,\bar{\rho}\circ\pi=\pi(\mbox{Im}\,\rho)=
	(\mbox{Im}\,\rho\cdot\BC^*I_2)/\BC^*I_2=
	\mbox{Im}\,\rho/(\mbox{Im}\,\rho\cap\BC^*I_2).
	\]
	Thus, to prove that Im$\,\rho$ is finite, we only need to prove that $\mbox{Im}\,\rho\cap\BC^*I_2$ is finite. Let $\G=\rho^{-1}(\BC^*I_2)=\mbox{Ker}\,\pi\circ \rho$, then $\rho$ acts as a character on $\G $ with $\rho(\gamma)=\chi(\gamma)I_2$. In the meantime, $\mbox{det}\,\rho$ is a character of $\SL$
	 of order dividing 12,  hence $\chi^{24}=1$. It follows that  $\rho(\G)=\mbox{Im}\,\rho\cap\BC^*I_2$ is finite if we can establish that $\G$ is finitely generated. To see this, as $\G=\mbox{Ker}\,\pi\circ \rho$, we have $\SL/\G\cong \mbox{Im}\,\pi\circ\rho =\mbox{Im}\,\bar{\rho}\circ\pi$ which is finite, and so $\G$ has a finite index in $\SL$ and thus finitely generated.
	\end{proof}
 We now start with a representation $\bar{\rho}$ of $\PS$ such that Im$\,\bar{\rho}$ is finite. We also assume that $\bar{\rho}$ has a lift $\rho$ to $\SL$ such that \eqref{cd} holds and hence $\rho$ has also a finite image according to the above proposition. As all our representations arise from equivariant functions, this will be always the case. Indeed,  if $h$ is $\bar{\rho}-$equivariant , then according to \cite[Theorem 6.6]{ss4}, then 
 $h=f_1/f_2$ where $F=[f_1,f_2]^t$ is a vector-valued modular form of multiplier $\rho$ and weight -1. In particular $\rho(-I_2)=-I_2$. 
 
  Assume first that $\rho$ is irreducible, then according to \cite{mason}, up to a twist by a character, $\rho$ is unitary and the order of $\rho(T)$ is not twice an odd number. Such a finite image representation is  referred to as a basic representation of $\SL$ in loc. cit. Since multiplying $\rho$ by a character does not affect $\bar{\rho}$ we can always assume that $\rho$ is a basic representation.
  \begin{thm}\cite[Theorem 3.5]{mason}
  	Let $\rho$ be a basic representation of $\SL$. Set $\G = \mbox{Ker}\,\rho$, $G = \SL/\G$, and
  	let $N$ be the order of $\rho(T )$. There are just four possibilities for $\G$, and one of the following holds:
  	\begin{enumerate}
  	\item $G$ is the binary dihedral group of order 12, $N = 4$.
  	\item $G$ is the binary tetrahedral group of order 24,  $N = 3$.
  	\item  $G$ is the binary octahedral group of order 48, $N = 8$.
  	\item  $G$ is the binary icosahedral group of order 120, $N = 5$.
  	  \end{enumerate}
  	In each case $\G(N)\subseteq \G$.
  \end{thm}
\begin{cor} Let $\rho$ be a basic representation of $\SL$ and $\G=\mbox{Ker}\,\rho$.
	If $N$ be the order of $\rho(T)$, then $\G(N)\subseteq \pm \G $ for $N\in\{3,4,5\}$.
\end{cor} 
\begin{proof} This is clear for the dihedral, tetrahedral and icosahedral cases. The octahedral case is explicit in the proof of the above theorem.
	\end{proof}
As a consequence, since the order of $\bar{\rho}(T)$ divides the order of $\rho(T)$, we have
\begin{prop}\label{prop4.4}
	If ${\rho}$ is an irreducible representation of $\SL$ with finite image and $N$ is the order of $\bar{\rho}(T)$, then $N\in\{2,3,4,5\}$.
\end{prop}

 \section{The Frobenius method}

In this section we focus on the case where $f$ in \eqref{schP} is a holomorphic automorphic form of weight 4 for the modular group $\SL$. As the space of such forms is one-dimensional, we are looking at the Schwarz differential equation
\begin{equation}\label{the-equ}
\{h,\tau\}\,=\,sE_4(\tau),
\end{equation}
where $s$ is a complex parameter and 
and
\[
E_4(\tau)\,=\,1+240\,\sum_{n\geq 1}\,\sigma_3(n)q^n\,,
\]
where $\sigma_3(n)$ is the sum of the cubes of the positive divisors of $n$ .
 If $s=0$ then $h$ is a linear fraction, and we will assume for the rest of this paper that $s\neq 0$. The  corresponding second degree ODE is given by
\[
y''+\frac{s}{2}\,E_4(\tau)\,y=0.
\]
Write $s=2\pi^2 r^2$, $r\in\BC$ with $\Re(r)\geq 0$, and set $\displaystyle q=e^{2\pi i\tau}$. The ODE becomes
\[
\frac{d^2 y}{{dq}^2}+\frac{1}{q}\frac{dy}{dq}-\frac{r^2}{4}\frac{E_4(q)}{q^2}\,y=0
\]
in the punctured disc $\{0<|q|<1\}$.
As $E_4(q)=1+\mbox{O}(q)$,  The ODE is a Fuchsian differential equation with a regular singular point at $q=0$. We apply the Frobenius method to determine the shape of the solutions.

The indicial equation for the ODE is given by
\[
x^2-\frac{r^2}{4}=0,
\]
and we set $x_1=r/2$ and $x_2=-r/2$ so that $\Re(x_1)\geq\Re(x_2)$. We always have a solution given by
\begin{equation}\label{sol-y1}
y_1(q)=q^{r/2}\,\sum_{n=0}^{\infty}\,c_n\,q^n\,,\ \mbox{ with } c_0\neq 0.
\end{equation}
The second solution depends on $x_1-x_2$.
If $x_1-x_2=r$ is not an integer, a linearly independent solution with $y_1(x)$ is given by
\begin{equation}\label{sol-y2}
y_2(q)=q^{-r/2}\,\sum_{n=0}^{\infty}\,c^*_n\,q^n\,,\ \mbox{ with } c^*_0\neq 0.
\end{equation}
In this case a solution to \eqref{the-equ} is given by
\[
h(\tau)=\frac{y_2(q)}{y_1(q)}=q^{-r}\,\sum_{n=0}^{\infty}\,a_nq^n\,,\ \mbox{ with } a_0\neq 0.
\]
If $x_2-x_1$ is an integer (which must be positive under our assumption), then a second solution is given by
\begin{equation}\label{sol-y1-log}
y_2(q)=k\log(q)y_1(q)+q^{-r/2}\,\sum_{n=0}^{\infty}\,C_n\,q^n\,,\ \mbox{ with } C_0\neq 0\,,\ k\in\BC.
\end{equation}
yielding a solution to \eqref{the-equ} of the form
\[
h(\tau)=k\log(q) + q^{-r}\,\sum_{n=0}^{\infty}\,b_nq^n\,,\ \mbox{ with } b_0\neq 0.
\]
The branch of the log is chosen so that $\log(q)=2\pi i \tau$ when $\tau\in\fH$. 
Since the set of solutions to \eqref{the-equ} is invariant under linear fractional transformations, in particular under inversion and under multiplication by a scalar, we deduce the following
\begin{prop} \label{prop5.1}
	We have
		\begin{enumerate}
			\item If $r\notin\BZ$, there is a solution to \eqref{the-equ} of the form
			\begin{equation}\label{sol-mero}
			h(\tau)=q^{r}\,\sum_{n=0}^{\infty}\,a_nq^n\,,\ \mbox{ with } a_0\neq 0.
			\end{equation}
			\item If $r\in\BZ$, then there is a solution to \eqref{the-equ} of the form
			\begin{equation}\label{sol-log}
			h(\tau)=\tau\, +\, q^{-r}\,\sum_{n=0}^{\infty}\,b_nq^n\,,\ \mbox{ with } b_0\neq 0.
			\end{equation}
		\end{enumerate}
\end{prop}
In these two cases, we keep in mind that any other solution is a linear fraction of the given one. The following is deduced easily from the above considerations.
\begin{prop}\label{prop5.2}
	Let $h$ be a solution to \eqref{the-equ} with $s=2\pi^2r^2$, and let $\rho$ be the 2-dimensional representation of $\SL$ attached to $h$, then  $T^m\in\ker\rho$ for some non-zero integer $m$ if and only if $r\in\BQ\setminus\BZ$. 
\end{prop} 
We end this section by providing the recurrence relations among the coefficients of the solutions $y_1(q)$ above. Write
\[
-\frac{r^2}{4}E_4=\sum_{n=0}^{\infty}\,\alpha_n\,q^n\,,\ \ \alpha_0=-r^2/4.
\]
It is easy to see that the coefficients $c_i$ of $q^{r+i}$ in $y_1(q)$ satisfy 
\[
[(r+s)^2+\alpha_0]c_s\,+\, \sum_{i=0}^{s-1}\alpha_{s-i}c_i=0\,,\ \ c_0\mbox{ being indeterminate}.
\]
Similar relations hold for the coefficients of $y_2(q)$ in the absence of the logarithmic term.
\section{Modular solutions: the degree 1 case} \label{unramified}
In this section, we  focus on  the solutions $h$ to \eqref{the-equ} whose attached representation $\rho$ is such that $\Gamma=\ker \rho$ is a finite index subgroup of $\SL$. In other words, $h$ is a modular function for $\Gamma$. Since $\G $ is normal in $\SL$, all its cusps have the same width $m$ which is defined as the smallest positive integer $m$ such that the matrix $\displaystyle \binom{1\ \ m}{0\ \  1}\in\G$ in the case of the cusp at infinity.  We will refer to the integer $m$ as the level of $\G$, even if $\G $ is not a congruence subgroup (This is Wohlfahrt's definition of the level as being the least common multiple of all cusp widths \cite{wohl} which coincides with Klein's definition of the level in the case of congruence subgroups). Moreover,
$q=exp(2\pi i\tau/m)$ is the local uniformizer at $\infty$ for the group $\G$. According to \propref{prop5.2}, we must have $s=2\pi^2r^2$ where $r\in\BQ\setminus\BZ$. If we write $r$ in lowest terms, then necessarily 
$r=n/m$ for a positive integer $n$. Thus we have
\begin{prop}
	If $h$ is a solution to \eqref{the-equ} and $\Gamma=\ker \rho$ is a finite index subgroup of $\SL$ then
	\begin{equation}
	s=2\pi^2\left(\frac{n}{m}\right)^2\,,
	\end{equation}
	where $n\in\BZ_{>0}$ and $m\in\BZ_{>0}$ is the level of $\Gamma$.
\end{prop}
Another way to look at this statement, which will be useful later, is that  $h$, as a modular function,
has a $q-$expansion meromorphic at $\infty$  which one can write from \propref{prop5.1}  as 
	\[
	h(z)=q^n+\sum_{i=n+1}^{\infty}\,a_iq^i\,,\ \ q=e^{\frac{2\pi i\tau}{m}}\,.
	\]
	In fact, \propref{prop5.1} says more: that $a_i=0$ if $i$ is not of the form $m+jn$, $j\geq 0$. Furthermore,
	an easy calculation using \eqref{sd2} with $w=q$ yields
	\[
	\{h,z\}=2\pi^2\left(\frac{n}{m}\right)^2\left[
	1+(-12a_3+12a_2^2)q^2+O(q^3)\right].
	\]
	Since $E_4(z)=1+240\;q^m+O(q^{2m})$, we deduce that $\displaystyle s=2\pi^2\left(\frac{n}{m}\right)^2$.
	
For the rest of this section, we will assume that $n=1$ which is equivalent to say that  $h$   has a simple zero at $\infty$.

\begin{thm}\label{genus0}
		If $h$ is a solution to \eqref{the-equ} having a simple zero at $\infty$  and $\Gamma=\ker \rho$ is a finite index subgroup of $\SL$ then $\G$ has genus zero.
\end{thm}
\begin{proof}
	Recall that a non-constant holomorphic mapping between two compact Riemann surfaces $f:S'\longrightarrow S$ is necessarily surjective and define a covering of finite degree $d$. If $g$ and $g'$ are the genera of $S$ and $S'$ respectively, then we have the Riemann-Hurwitz formula \cite{shimura}
	\[
	2g'-2=d(2g-2) +\sum_{z\in S'}\,(e_z-1),
	\]
	where $e_z$  is the ramification index of the covering at $z$. Now,
	suppose that $\Gamma=\ker \rho$ is a finite index subgroup of $\SL$. As $\{h,z\}$ is holomorphic on $\fH$, it follows that $h$ takes only simple values and simple poles on $\fH$. Moreover, $h$ takes either a simple value or has a simple pole at any  cusp $c$ of $\G$. Indeed,  let $\gamma\in\SL$ such that $\gamma\cdot \infty=c$ and set $g=h\circ\gamma$. The meromorphic behavior of $h(\tau)$ at $c$ is given by the meromorphic behavior of $g$ at $\infty$. Notice that since $\G $ is normal in $\SL$, $g$ is also invariant under $\G$. Furthermore,
	we have $g(\tau)=\rho(\gamma) h(\tau)$. Therefore, since $h$ takes a simple zero at $\infty$, $g(\tau)$ either takes a simple value  or has a simple pole at $\infty$ as $\rho(\gamma)$ is an automorphism of the Riemann sphere.  If $\fH^*=\fH\cup\BQ\cup\{\infty\}$ and $X(\G)=\G\backslash \fH^*$, then we have a covering of Riemann surfaces $h:X(\G)\longrightarrow {\mathbb P}_1(\BC)$ that is nowhere ramified. If $d$ is the degree of this covering, the Riemann-Hurwitz formula yields $2g'-2=-2d$ where $g'$ is the genus of $X(\G )$. It follows that $d=1$ and $g'=0$.
	\end{proof}
\begin{remark}
	The same argument as in the above proof shows that if $\Gamma=\ker \rho$ has finite index in $\SL$, then $h$ has the same ramification index at all cusps, even when this index is not 1.
\end{remark}
\begin{prop}
	If  $h$ is a solution to \eqref{the-equ} then $\Gamma=\ker \rho$  is torsion-free, that is, it has no elliptic elements.
\end{prop}
\begin{proof} Suppose that $\G $ has an elliptic element 
$\displaystyle \gamma=\mat$ and $\tau_0\in\fH$ is an elliptic fixed point by 
	$\gamma$.  For $\tau \in\fH$, differentiating
	$h(\gamma\cdot\tau)=h(\tau)$ yields $h'(\gamma\cdot\tau)=(c\tau+d)^2h'(\tau)$. This gives, for $\tau=\tau_0$, $h'(\tau_0)=(c\tau_0+d)^2h'(\tau_0)$. As $h'(\tau_0)\neq 0$, we must have $(c\tau_0+d)^2=1$ which implies that $\tau_0$ is real, which is impossible. 
	\end{proof}

It turns out that there are only finitely many possibilities for such subgroups.
\begin{prop}\label{4.4}	If $h$ is a solution to \eqref{the-equ} having a simple zero at $\infty$ and $\Gamma=\ker \rho$ is a finite index subgroup of $\SL$ of level $m$, then $2\leq m\leq 5$.
		\end{prop}
\begin{proof}
	If $\G$ is a finite index subgroup of $\SL$, we  have  the natural covering of compact Riemann surfaces 
	$X(\G )\longrightarrow X(1)$ where $X(1)=X(\SL)$, for which the Riemann-Hurwitz formula reads \cite{rankin}
	\begin{equation}\label{r-h}
	g\,=\, 1+\frac{\mu}{12}-\frac{e_2}{4}-\frac{e_3}{3}-\frac{\nu_{\infty}}{2},
	\end{equation}
	where $g$ is the genus of $X(\G)$, $\mu=[\PS:\bar{\G}]$,  $e_k$, $k\in\{2,3\}$, is the number of inequivalent elliptic fixed points of order $k$, and $\nu_{\infty}$ is the parabolic class number, that is the number of $\G-$inequivalent cusps. In our case, $\G$ is genus 0 and torsion-free so that the Riemann-Hurwitz becomes
	\begin{equation}\label{rh0}
	\mu=6(\nu_{\infty}-2).
	\end{equation}
	In the meantime, the degree of the above covering is $\mu$ and it is also the sum of the ramification indices of the $\nu_{\infty}$ cusps above $\infty$. These ramification indices are all equal to $m$ as $\G $ is normal in $\SL$. In other words, we have $\mu=m\nu_{\infty}$. Hence we have $m\nu_{\infty}=6(\nu_{\infty}-2)$ and so $m<6$ and $(6-m)\nu_{\infty}=12$. We conclude that necessarily $2\leq m\leq 5$.
	\end{proof}
Define, for  $n\in\BZ$, the group 
\[
\Delta(n)=\langle \gamma T^n\gamma^{-1}\,,\ \gamma\in \SL\rangle.
\]
If $m$ is the (Wohlfahrt) level of a finite index normal subgroup of $\SL$, then $\Delta(m)$ is a subgroup of $\G $ and we have the covering $X(\Delta(m))\longrightarrow X(\G)$. This is true in particular for $\G=\G(m)$, the principal congruence subgroup of level $m$. One could argue that this covering is, in fact, the universal covering of $X(\G(m))$ \cite{wohl}. However,, when  $X(\Gamma(m))$ has a positive genus, its universal covering is infinite-sheeted and this occurs if and only if $m\geq 6$. To be more precise, we have
\begin{prop}\label{4.5}\cite{rankin} We have  $\Delta(m)=\G (m)$ if $1\leq m\leq 5$ , and for $6\leq m$, we have  $[\G(m):\Delta(m)]=\infty$.
\end{prop}
Finally, we have
\begin{thm}
		If $h$ is a solution to \eqref{the-equ} such that $\ker \rho$ is a finite index subgroup of $\SL$, then $\ker \rho$ is one of the four groups $\G (m)$, $2\leq m\leq 5$.
\end{thm}
\begin{proof}
	Let $h$ be a solution to \eqref{the-equ} such that  $\Gamma_m=\ker \rho$ is a finite index subgroup of $\SL$ of level $m$. According to \propref{4.4}, we have $2\leq m\leq 5$. In the meantime, as $\Gamma_m$ is normal, it contains $\Delta(m)=\Gamma(m)$. In other words, for $2\leq m\leq 5$, we have $\G(m)\subseteq \G_m$. Both groups are genus zero and torsion-free, and so can be generated by parabolic elements only. As the cusps for both groups have the same  width which is $m$, the two groups share the same parabolic elements, and thus must be equal.
	\end{proof}
\begin{cor}
		If $h$ is a solution to \eqref{the-equ} such that $\ker \rho$ is a finite index subgroup of $\SL$ then $h$ is a Hauptmodul for $\ker \rho$.
\end{cor}
\begin{proof} This can be readily seen from the proof of \thmref{genus0} as the covering $h:X(\G )\longrightarrow  {\mathbb P}_1(\BC)$ is an isomorphism, however, for the sake of completeness, we will provide a more direct proof.
	Let $m$ be the level of $\ker\rho$ and let $f_m$ be a Hauptmodul for $\Gamma(m)$,  $2\leq m\leq 5$. Then $\{f_m,\tau\}$ is a weight 4 holomorphic modular form for the normalizer of $\G(m)$ which is $\SL$. Hence $f_m$ is a solution to \eqref{the-equ} with  $\displaystyle s=\frac{2\pi^2}{m^2}$ as $m$ is also the width at $\infty$ in $\G (m)$. Now, since $h$ and $f_m$ have the same Schwarz derivative,  $h$ is a linear fraction of $f_m$, and so it is also a Hauptmodul for $\G (m)$. 
\end{proof}

The following classical  modular forms and functions  will  be useful for the rest of this paper:

For $\tau\in\fH$, set $\disp q=e^{2\pi i \tau}$ and $\disp t=e^{\pi i\tau}$ so that $t^2=q$.
The Dedekind $\eta-$function is defined by
\[
\eta(\tau)\,=\,q^{\frac{1}{24}}\,\prod_{n\geq1}\,(1-q^n)\,,
\]
and the discriminant $\Delta$ is defined by
\[
\Delta(\tau)=\eta(\tau)^{24}=q\,\prod_{n=1}^{\infty}(1-q^n)^{24}\,.
\]
We also have the Eisenstein series
\[
E_2(\tau)=\frac{1}{2\pi i}\frac{\Delta'}{\Delta}=1-24\,\sum_{n=1}^{\infty}\,\sigma_1(n)\,q^n\,,
\]
with $\sigma_k(n)$ being the sum of the $k$-th powers of the positive divisors of $n$.
The Jacobi (null-) theta functions are given by
\[
\theta_2(\tau)=2\,\sum_{n=0}^{\infty}\,t^{(n+\frac{1}{2})^2}\,=\,
\frac{2\eta(4\tau)^2}{\eta(2\tau)}\,,
\]
\[
\theta_3(\tau)=1+2\,\sum_{n=0}^{\infty}\,t^{n^2}\, =\,
\frac{\eta(2\tau)^5}{\eta(\tau)^2\eta(4\tau)^2}\,,
\]
and
\[
\theta_4(\tau)=1+2\,\sum_{n=0}^{\infty}\,(-1)^n t^{n^2}\,=\,
\frac{\eta(\tau)^2}{\eta(2\tau)}\,.
\]
The elliptic modular functions $\la$ and $j$ are given by
\[
\la=\frac{\theta_2^4}{\theta_3^4}\ ,\quad j=\frac{E_4^3}{\Delta}.
\]
The modularity of the above functions will be invoked subsequently when needed.
 From   \cite{ss2}, we have the following expressions for the hauptmoduln $f_m$ which thus provide, up to a linear fractional transformation, the only solutions to \eqref{the-equ} with a finite index kernels and which are unramified at the cusps:
\[
f_2(\tau)=\lambda(\tau),\ 
f_3(\tau)=\left(\frac{\eta(3\tau)}{\eta(\tau/3)}\right)^3,
\]
\[
f_4(\tau)=\frac{\eta(\tau/2)\eta(4\tau)^2}{\eta(\tau/4)^2\eta(2\tau)},\ 
f_5(\tau)=q^{\frac{1}{5}}\prod_{n\geq1}\,(1-q^n)^{\left(\frac{n}{5}\right)},
\]
where $\left(\frac{\,\cdot\,}{\,\cdot\,}\right)$ is the Legendre symbol.
These Hauptmoduln are $\rho-$equivariant for $\SL$ and the representations are determined by how each one transforms under the coset generators of the subgroups. As an example, for $\Gamma(2)$, the representation is trivial on $\Gamma(2)$ and on the  coset generators, it is determined by:
\begin{equation}\label{la-transf}
\lambda(\tau+1)=\frac{\lambda(\tau)}{\lambda(\tau)-1}\,,\ \lambda(-1/\tau)=1-\lambda(\tau)\,,\ 
\lambda\left(\frac{\tau}{\tau+1}\right)=\frac{1}{\lambda(\tau)}\,,
\end{equation}
\[
\lambda\left(\frac{-1}{\tau+1}\right)=\frac{1}{1-\lambda(\tau)}\,,\ 
\lambda\left(\frac{1+\tau}{-\tau}\right)=1-\frac{1}{\lambda(\tau)}\,.
\]
\begin{remark}
	The results of this section show that all the cases of \propref{prop4.4} for an irreducible representation $\rho$ occur.
\end{remark}

We end this section with an application at the level $m=2$. 
\begin{prop}
	The differential equation
	\begin{equation}\label{odel2}
	y''+\frac{\pi^2}{4}\;E_4\;y=0
	\end{equation}
	has two linearly independent solutions given by
	\[
	y_1=\frac{\theta_2^2}{\theta_3^2\theta_4^2}\,, \quad 
	y_2=\frac{\theta_3^2}{\theta_2^2\theta_4^2}.
	\]
\end{prop}
\begin{proof}
	The solution to the Schwarzian differential equation corresponding to \eqref{odel2} is $\lambda=\theta_2^4/\theta_3^4$. Comparing weight 2 and level 2 modular forms yields \cite[Formula 7.2.12]{rankin}
	\[
	i\pi \theta_2^4=\frac{\lambda'}{1-\lambda}\,,\ \ 
	i\pi \theta_3^4=\frac{\lambda'}{\lambda(1-\lambda)}\,,\ \ 
		i\pi \theta_4^4=\frac{\lambda'}{\lambda}.	
	\]
	Therefore, we get
	\[
	\lambda'=i\pi \theta_4^4\lambda=i\pi\frac{\theta_4^4\theta_2^4}{\theta_3^4},
	\]
	and we choose
	\[
	\sqrt{\lambda'}=(i\pi)^{1/2} \frac{\theta_4^2\theta_2^2}{\theta_3^2},
	\]
	so that
	\[
	\frac{\lambda}{\sqrt{\lambda'}}=(i\pi)^{-1/2}\frac{\theta_2^2}{\theta_3^2\theta_4^2}\ \mbox{ and }\ \frac{1}{\sqrt{\lambda'}}=
	(i\pi)^{-1/2}\frac{\theta_3^2}{\theta_2^2\theta_4^2}
	\]
	yielding the desired expressions for $y_1$ and $y_2$.
	\end{proof}
It is worth noting that, using the Jacobi identity $\theta_3^4=\theta_2^4+\theta_3^4$, we get
\[ y_1-y_2=\frac{\theta_4^2}{\theta_2^2\theta_3^2}
\]
which is also a solution to \eqref{odel2}.

\section{Modular solutions: the ramified  case} \label{ramified}

In this section we continue our investigation of solutions $h$ to \eqref{the-equ} with the corresponding  representation $\rho$ having a finite index kernel $\G $ in $\SL$ and we allow $h$ to be ramified at the cusps. In this case we have $ \{h,\tau\}=sE_4$ with $s=2\pi^2(n/m)^2$,  $\gcd(m,n)=1$ and $n\geq 1$. Here, $m$ is the level of $\G$ and  if we choose $h$ to vanish at $\infty$ then $n$ is its order of vanishing. Denote by $\mbox{C}(\G)$ the set of inequivalent cusps modulo $\G $ so that  $\nu_{\infty}=|\mbox{C}(\G)|$ is the cusp number of $\G$. We consider the covering
\[
h:X(\G)\longrightarrow \PC, 
\]
and let $d$ be its degree. We have seen that $h$ is only ramified at the cusps with the same ramification index which is $n$. The Riemann-Hurwitz formula reads:
\[
2g-2=-2d+\sum_{c\in\mbox{C}(\G)}\,(n-1)
	=-2d+\nu_{\infty}(n-1)
\] 
which we rewrite as
\begin{equation}\label{rh1}
2g-2+\nu_{\infty}=-2d+n\nu_{\infty}.
\end{equation}
In the meantime, we have the natural covering induced by the inclusion $\G\subseteq\SL$:
\[
X(\G )\longrightarrow\PC
\]
for which the Riemann-Hurwitz formula reads as 
\begin{equation}\label{rh2}
\frac{\mu}{6}= 2g-2+\nu_{\infty},
\end{equation}
as $\G $ is torsion-free and where $\mu=[\PS,\bar{\G }]$ is also the degree of this covering. The second covering is also ramified exactly at the cusps and the ramification indices are all equal to the cusp width $m$ since $\G$ is normal in $\SL$.  In particular, we have  $ \mu=m\nu_{\infty}$.
Essentially, the integer $m$ plays the same role for the second covering that $n$ plays for the first one. 
Equating \eqref{rh1} and \eqref{rh2} proves the following
\begin{prop}
	Let  $h$ be a solution to \eqref{the-equ} such that $\G$ has a finite index in $\SL$ with  $s=2\pi^2(n/m)^2$,  $\gcd(m,n)=1$. Let $d$ be the degree of $h$ as a covering map. Then
	\begin{equation}\label{deg-index}
	(6n-m)\nu_{\infty}=12d.
	\end{equation} 
\end{prop}
We assume for the remainder of this section  that $\rho$ is irreducible. According to \secref{reps}, the level $m$ satisfies
 $2\leq m\leq 5$ so that $\G =\G (m)$. 

We can state
\begin{thm} \label{thmp1}
	Let $m$ be an integer such that $2\leq m\leq5$ and $n$ such that $\gcd(m,n)=1$. The solutions to
	\[
\{h,\tau\}=2\pi^2\left(\frac{m}{n}\right)^2\,E_4	
	\]
	are modular functions for $\G (m)$. Moreover, each solution is a rational function of a Hauptmoduln for $\G(m)$ of degree $d$ satisfying $(6n-m)\nu_{\infty}=12d$.
\end{thm}
For each $\G(m)$, the cusp number is given by \cite{rankin}
\[
\nu_{\infty}=\frac{\mu}{m}=\frac{m^2}{2}\prod_{p\mid n,\, p \,{\tiny\mbox{prime}}}\left(1-\frac{1}{p^2}\right).
\]
The following table illustrate the relation \eqref{deg-index} between the degree $d$ and the ramification index $n$ for the values of $m$ of interest.

\begin{center}
\begin{tabular}{|c|c|}
	\hline
	$m$& Relation\\
	\hline
&	\\
	2&  $2d+1=3n$  \\  &\\
	3&   $d+1=2n$  \\  &\\
	4&    $d+2=3n $ \\  &\\
	5&     $d+5=6n  $\\
	\hline
\end{tabular}
\end{center}
In the following, $t$ will indicates a Hauptmodul for $\G(m)$, and we provide solutions for few values of the triples $(m,n,d)$.
\begin{prop}
	Let $t=\lambda$ be a Hauptmodul for $\G(2)$, then
	\[
	h=\frac{t^3(t-2)}{t^4-2t^3+4t-2}
	\]
	is  a solution to $\displaystyle \{h,\tau \}=2\pi^2(3/2)^2\,E_4$.
\end{prop}
\begin{proof}
Since $\la(\infty)=0$, we can deduce the values at the other two cusps 0 and 1 from the relations \eqref{la-transf}: $\la(0)=1$ and $\la (1)=\infty$. From the same relations we obtain:
\[
h(\tau+1)=-h(\tau)\,,\  h(-1/\tau)=\frac{h(\tau)+1}{3h(\tau)-1}\,,\  
h\left(\frac{\tau}{\tau+1}\right)=\frac{-h(\tau)+1}{3h(\tau)+1}\,,
\]
\[
h\left(\frac{-1}{\tau+1}\right)=\frac{h(\tau)-1}{3h(\tau)+1}\,,\ 
h\left(\frac{1+\tau}{-\tau}\right)=\frac{h(\tau)+1}{-3h(\tau)+1}\,.
\]
Therefore $h$ is $\rho-$equivariant with $\rho$ trivial on $\G (2)$ and given by the above formulas on the coset generators. Furthermore,
\[
h'(\tau)=\frac{12t^2(t-1)^2\,t'}{(t^4-2t^3+4t-2)^2},
\]
so that $h'$ does not vanish on $\fH$ as the values 0 and 1 are attained at the cusps $\infty$ and 0 respectively and $t'=\lambda'$ does not vanish on $\fH$. As $t=16\;q^{1/2}-128q+\mbox{o}(q)$ we see that $h=4096\;q^{3/2}+\mbox{o}(q^{3/2})$.  Therefore $\{h,\tau \}$ is a weight 4 modular form for $\SL$ which is holomorphic on $\fH$ and at the cusps and clearly  $\displaystyle \{h,\tau \}=2\pi^2(3/2)^2\,E_4$.
	\end{proof}
It is worth mentioning that $h'/\lambda'$ is readily a square, and so it is possible to write down explicit solutions to the ODE $\displaystyle y''+2\pi^2(3/2)^2\,E_4\,y=0$.

In the following table we will provide other solutions for small values of $n$. The proof is essentially similar to the level 2 case but more tedious. Here again $t$ stands for a Hauptmodul for the relevant $\G (m)$.
\newpage
{\small\begin{center}
	\begin{tabular}{|l|c|}
		\hline
		&\\
		$\disp (m,n,d)$& $\disp h$\\
		&\\
		\hline
			&\\
	$(2,3,4)$&$\displaystyle\frac{t^3(t-2)}{t^4-2t^3+4t-2}$\\
	&\\
	\hline
		&\\
	$(2,5,7)$&$\displaystyle\frac {{t}^{5} \left( 2\,{t}^{2}-7\,t+7 \right) }{ \left( t-2 \right)  \left( 2\,{t}^{6}-3\,{t}^{5}+{t}^{4}+2\,{t}^{3}+4\,{t}^{2}-6\,t+2 \right) }$\\
	&\\
	\hline
		&\\
	$(3,2,3)$&  $\displaystyle \frac {{t}^{2} \left( t+1 \right) }{26\,{t}^{3}+26\,{t}^{2}+27\,t+9}$\\
	&\\
	\hline
		&\\
	$(3,4,7)$& $\displaystyle \frac {{t}^{4} \left( 9\,{t}^{3}+21\,{t}^{2}+21\,t+7 \right) }{ \left( t+1 \right)  \left( 27\,{t}^{6}+36\,{t}^{5}+27\,{t}^{4}-6\,{t}^{3}-15\,{t}^{2}-6\,t-1 \right) } $ \\
	&\\
	\hline
		&\\
	$(4,3,7)$&$\displaystyle   \frac {{t}^{3} \left( {t}^{4}+7\,{t}^{3}+21\,{t}^{2}+28\,t+14 \right) }{ \left( {t}^{4}+{t}^{3}+3\,{t}^{2}+4\,t+2 \right)  \left( t+2 \right) ^{3}} $  \\
	&\\
	\hline
		&\\
		$(4,5,13)$&$\displaystyle \frac {{t}^{5} \left( {t}^{8} +13\,{t}^{7}+78\,{t}^{6}+286\,{t}^{5}+689\,{t}^{4}+1092\,{t}^{3}+1092\,{t}^{2}+624\,t+156\right) }{ \left( t+2 \right)   ^{5}  \left( {t}^{8}+3\,{t}^{7}+8\,{t}^{6}+6\,{t}^{5}-11\,{t}^{4}-28\,{t}^{3}-28\,{t}^{2}-16\,t-4 \right) }$\\
		&\\
		\hline
		&\\
	$(5,2,7)$&$\displaystyle  \frac {{t}^{2} \left( {t}^{5}-7 \right) }{7\,{t}^{5}+1} $  \\
	&\\
	\hline
		&\\
	$(5,3,13)$&$\displaystyle  \frac {{t}^{3} \left( {t}^{10}-39\,{t}^{5}-26 \right) }{26\,{t}^{10}-39\,{t}^{5}-1}  $  \\
		&\\
		\hline
		
	\end{tabular}
\end{center}
}
\section{The reducible case}

We now go back to our equation \eqref{the-equ} and let $h$ be a solution and $\rho$ its 2-dimensional representation. If $\ker \rho$ has finite index in $\SL$ and $\rho$ is irreducible, then according to \secref{reps}, we are in one of the cases  of \secref{unramified} and \secref{ramified}. In this section we will focus on the case of $\rho$ being reducible. There exists a constant matrix $\sigma\in\GC$ such $\sigma\rho\sigma^{-1}$ is (upper)  triangular, and according to \propref{sigma}, it corresponds to the equivariant function $g=\sigma\cdot f$ which satisfies the same Schwarzian equation as $f$. In addition, the kernels of both representations are the same. Therefore, without loss of generality, we will assume that $\rho$ is triangular.

Now we turn our attention to the question whether there exists a solution with the representation $\rho$ being triangular and $\ker\rho$ having a finite index in $\SL$. 
\begin{lem}
	If $h$ is a solution with $\rho$ reducible and $\ker\rho$ is of finite index, then $\ker\rho=\G'$, the commutator group of $\SL$.
\end{lem}
\begin{proof}
Suppose that $h$ is a such solution with $\rho$  upper-triangular. Write, for $\gamma\in\SL$, $\disp \rho(\gamma)=\binom{a(\gamma)\ \ b(\gamma)}{\ 0\quad\ \   d(\gamma)}$. Then $a$ and $b$ are characters of $\SL$ and so they are trivial on the commutator $\G'$. In other words, for each $\gamma\in\G'$, $\rho(\gamma)$ is unipotent and since it has finite order (as Im$\,\rho$ is finite) it must be the identity. Therefore we have $\G'\subseteq\ker\rho$. Recall that $\G'$ is a genus 1, level 6, torsion-free and normal subgroup of $\SL$ of index 12.  In the meantime, according to \cite[Page 36]{rankin}, there are only 7 proper normal subgroups between $\ol{\G}(6) $ and $\PS$ as shown in the figure
	\begin{center}
	\begin{tikzpicture}
	\draw (0,8) circle (0.5) node {$\ol{\G}(1)$};
	\draw (-1.65,6.35)--(-0.35,7.65); 	\draw (1.65,6.35)--(0.35,7.65);	
	\draw (-2,6) circle (0.5) node {$\ol{\G }^2$};
	\draw (2,6) circle (0.5) node {$\ol{\G }^3$};
	\node at (-1.2,7.1){2};\node at (1.2,7.1){3};
	\draw (-3.65,4.35)--(-2.35,5.65); 	\draw (-0.35,4.35)--(-1.65,5.65);
	\draw (0.35,4.35)--(1.65,5.65);\draw (3.65,4.35)--(2.35,5.65);
	\draw (-4,4) circle (0.5) node {$\ol{\G }(2)$};
	\draw (0,4) circle (0.5) node {$\ol{\G }'$};
	\draw (4,4) circle (0.5) node {$\ol{\G }(3)$};
	\node at (-3.2,5.1){3};\node at (-0.8,5.1){3};
	\node at (0.8,5.1){2};\node at (3.2,5.1){4};
	\draw (-2.35,2.35)--(-3.65,3.65); 	\draw (-1.65,2.35)--(-0.35,3.65);
	\draw (1.65,2.35)--(0.35,3.65);\draw (2.35,2.35)--(3.65,3.65);
	\node at (-2.8,3.1){3};\node at (-1.2,3.1){3};
	\node at (1.2,3.1){4};\node at (3.2,2.9){2};
	\draw (-2,2) circle (0.5) node {$\ol{\G ^2}'$};
	\draw (2,2) circle (0.5) node {$\ol{\G ^3}'$};
	\draw (-0.35,0.35)--(-1.65,1.65); \draw (0.35,0.35)--(1.65,1.65);
	\node at (-0.8,1.1){4};\node at (1.2,0.9){3};
	\draw (0,0) circle (0.5) node {$\bar{\G }(6)$};
	{\tiny \node at (0.5,-1){Normal subgroups between $\ol{\G}(6) $ and $\PS$};}
	\end{tikzpicture}
\end{center}
In particular, there are only two proper normal subgroups containing the commutator, namely $\bar{\G} ^2$ of level 2 and index 2  generated of the squares of elements of $\PS$ and the subgroup $\bar{\Gamma}^3$ of level 3 and index 3 generated by the cubes of elements of $\PS$. However, none of these 2 groups is torsion-free as can be easily seen from the Riemann-Hurwitz formula \eqref{r-h}. Therefore, $\ker\rho=\G'$.
\end{proof}

This leads us to the following conclusion.
\begin{thm} There is no solution to \eqref{the-equ} such that $\rho$ is reducible and $\ker\rho$ having a finite index in $\SL$.
\end{thm}
\begin{proof}
	Suppose there is such $\rho-$equivariant function $h$. According to the above lemma $\ker\rho =\G '$ whose image  in $\PS$ has index $\mu=6$, level $m=6$ and genus $g=1$. Recall that $\mu=m\nu_{\infty}$ so that $\G '$ has only one cusp. Let $d$ be the degree of the covering $h: X(\G ')\longrightarrow\BP_1$ which is unramified on $\fH$ and (possibly) ramified at the lone cusp $\infty$.
	Combining \eqref{rh1} and \eqref{rh2}, which are still valid in our case, yields
\[
n=2d+1,
\]
where $n$ is the ramification index of $h$ at $\infty$. However, this is impossible since $d\geq n$. 
\end{proof}

In conclusion, to supplement \thmref{thmp1} without the irreducibility condition, we can state
\begin{thm}
	The solutions to  \eqref{the-equ} are modular functions if and only if $s=2\pi^2(n/m)^2$ with $m$ and $n$ two positive integers satisfying $2\leq m\leq 5$ and $\gcd(m,n)=1$ in which case the invariance group of the solutions is given by $\G(m)$.
\end{thm}


\end{document}